\documentclass[a4paper]{amsart}
\usepackage[T1]{fontenc}
\usepackage{amssymb}
\usepackage{hyperref}
\usepackage{tikz}
\usepackage{soul}

\newcommand{\N}{\mathbb{N}}
\newcommand{\R}{\mathbb{R}}
\newcommand{\F}{\mathcal{F}}

\newcommand{\diam}{\text{diam}}

\theoremstyle{plain}
\newtheorem{tw}{Theorem}[section]
\newtheorem{lem}[tw]{Lemma}

\newtheorem{obs}[tw]{Proposition}
\newtheorem{problem}[tw]{Problem}

\theoremstyle{definition}
\newtheorem{df}[tw]{Definition}
\newtheorem{prz}[tw]{Example}

\theoremstyle{remark}
\newtheorem{rem}[tw]{Remark}

\begin{document}

\begin{abstract}
    Main aim of this article is to prove that for any continuous function $f \colon X \to X$, where $X$ is metrizable (or, more generally, for any family $\F$ of such functions, with one extra condition), there exists a compatible metric $d$ on $X$ such that the nth iteration of $f$ (more generally, composition of any $n$ functions from $\F$) is Lipschitz with constant $a_n$ where $(a_n)_{n=1}^{\infty}$ is an arbitrarily fixed sequence of real numbers such that $1 < a_n$ and $\lim\limits_{n\to+\infty}a_n = +\infty$. In particular, any dynamical system can be remetrized in order to significantly control the distance between points by their initial distance.
\end{abstract}
\title[Remetrizing dynamical systems]{Remetrizing dynamical systems to control distances of points in time}
\author[K.\ Go\l{}\k{e}biowski]{Krzysztof Go\l{}\k{e}biowski}
\date{}

\address{ Wydzia\l{} Matematyki i~Informatyki\\Uniwersytet Jagiello\'{n}ski\\ ul.\ \L{}ojasiewicza 6\\30-348 Krak\'{o}w\\Poland}
\email{krzysztof.golebiowski@student.uj.edu.pl}


\maketitle


\section{Introduction}

One of the important parts of the theory of dynamical systems is to control the distance between points as they move in time. It concerns notions such as \textit{expansivity}, \textit{proximality} or \textit{distality} (see Chapter IV, Section 2 in \cite{vries}). It may be useful to consider the distance between points with respect to their initial distance which leads to natural questions such as which families may be made Lipschitz after remetrization, Lipschitz with a priori given constants, or Lipschitz for next iterations of members of these families. An example could be a dynamical system consisting of one function $f$ and we could be interested in how small the Lipschitz constants of $f^n$ could be. The main result of this paper is regarding an arbitrary family with an additional property (finite equicontinuity; see Definition \ref{skladanie} in Section 2). In order to illustrate the main theorem, we present a special case for a dynamical system consisting of a single function:
\begin{tw}
    For a continuous function $f : X \to X$ on a metrizable space $X$ there exists a compatible metric $d$ such that $f^n$ is Lipschitz (w.r.t. $d$) and has Lipschitz constant not greater than $\log(n+2)$ for any $n > 0$.
\end{tw}

\begin{tw}
    For any homeomorphism $h$ on a metrizable space $X$ there exists a compatible metric $d$ such that for any $k \in \mathbb{Z} \setminus \{0\}$ the homeomorphism $h^k$ satisfies the Lipschitz condition with a constant not exceeding $\log(|k|+2)$.
\end{tw}

\begin{tw}
    Let $G$ be a finitely generated transformation group of homeomorphisms of a metrizable space $X$ and let $h_1, \ldots, h_p$ be generators of $G$. Denote by $m$ the length metric on $G$ induced by the generating set $S = \{h_1,...,h_p,id,h_1^{-1}, \ldots, h_p^{-1}\}$; that is, for each $g \in G \setminus \{id\}$, $m(g)$ is the least natural number $n$ such that $g$ can be expressed as the product of $n$ elements of $S$. Then there exists a compatible metric $d$ on $X$ such that each map $g$ from $G \setminus \{id\}$ satisfies the Lipschitz condition (w.r.t. d) with a constant not greater than $\log(m(g)+2)$.
\end{tw}

In fact, the Lipschitz constants that appear above may tend to infinity arbitrarily slowly (see Theorem \ref{glowne_twierdzenie}). Hence we see we can control the distance between points with respect to their initial distance with arbitrarily given Lipschitz constants. 

The main result is Theorem \ref{glowne_twierdzenie}, which
deals with the way of controlling the distance between points with respect to initial distance formulated not only for Lipschitz constants but more generally for moduli of continuity (see the next subsection), and theorems stated above. Moreover, we give an interesting example regarding the notion of finite equicontinuity (Example \ref{przykladzik}).

We hope these results will give new tools to work with dynamical systems, in particular with those defined on noncompact spaces.

Basic information about dynamical systems may be found in \cite{furst}, \cite{vries} and \cite{aoki}.

\subsection*{Notation and terminology}

In this article, by $\R_+$ we will denote the set $[0,+\infty)$. $\N$ will denote the set of all nonnegative integers. By $\diam_d(X)$ we will denote the diameter of a space $X$ in metric $d$, that is, $\diam_d(X)=\sup\{d(x,y) : x,y \in X\}$. $X^Y$ denotes the set of all functions from $Y$ into $X$.

A function $\omega: \R_+ \to \R_+$ will be called a \textit{modulus of continuity} if:
\begin{itemize}
    \setlength\itemindent{30pt}
    \item[(MC1)] $\omega(0) = 0 = \lim\limits_{t\to 0^+}\omega(t)$,
    \item[(MC2)] $\omega$ is nondecreasing; that is, $\omega(x) \leq \omega(y)$ whenever $0 \leq x \leq y$,
    \item[(MC3)] the function $\R_+\setminus\{0\} \ni t \mapsto \frac{\omega(t)}{t} \in \R_+$ is monotone decreasing; that is, $\frac{\omega(x)}{x} \geq \frac{\omega(y)}{y}$ whenever $0 \leq x \leq y$,
\end{itemize}

\noindent Note the limit $\lim\limits_{t\to 0^+}\frac{\omega(t)}{t}$ exists in $\R\cup\{+\infty\}$. We will call that limit \textit{the derivative at 0 of $\omega$} and it will be denoted by $\omega'(0)$. It follows that if $\omega'(0) \in \R_+$ then $\omega(t)\leq\omega'(0)t$ for all $t \in \R_+$. Moreover, moduli of continuity are subadditive, that is, $\omega(x+y) \leq \omega(x) + \omega(y)$ for $x,y \in \R_+$. The set of all moduli of continuity will be denoted by $MC$.

The reader might wonder why we take such a definition while e.g. N. Aronszajn and P. Panitchpakdi (see \cite{aron}) define them in a slightly weaker way, namely using only conditions (MC1) and (MC2). The reason is that in our definition, $MC$ is a convex set which is closed under addition, taking pointwise infima and suprema of bounded above sets (by a modulus of continuity). By the last two notions we mean that for two sets $A,B \subset MC$ (such that there exists $\tilde{\omega} \in MC$ for which $\omega < \tilde{\omega}$ whenever $\omega \in B$) the following functions are well defined moduli of continuity: $\inf A(t):=\inf\{\omega(t) : \omega \in A\}$ and $\sup B(t):=\sup\{\omega(t) : \omega \in B\}$. Moreover, any function $\omega$ satisfying conditions (MC1), (MC2) and additionally $\limsup\limits_{t\to+\infty}\frac{\omega(t)}{t} < +\infty$ is bounded from above by a concave function from $MC$ (see \cite{aron} for a proof). Functions satisfying (MC2) and (MC3) which vanish at 0 (so the only difference from our moduli of continuity is the lack of continuity at 0) are known as quasiconcave (see \cite{colin}).

\indent For the reader's convenience we remind a classical notion.
For a topological space $X$, a metric space $(Y,\rho)$ and a family $\F \subset Y^X$ we will say that the family $\F$ is \textit{equicontinuous (with respect to the metric $\rho$)} if for any $x\in X$ and any $\varepsilon > 0$ there exists a neighbourhood $U$ of $x$ such that for any function $f \in \F$ we have $\diam_\rho(f(U)) \leq \varepsilon$. When $X$ is a metric space with metric $d$ then we will say that $\F$ is \textit{uniformly equicontinuous (with respect to the metrics $d$ and $\rho$)} if for any $\varepsilon > 0$ there exists $\delta > 0$ such that for any $x,y \in X$ and $f \in \F$ if $d(x,y) < \delta$ then $\rho(f(x),f(y)) < \varepsilon$. Note that the union of finitely many [uniformly] equicontinuous families is [uniformly] equicontinuous.\\
\indent Let $(X,d)$ and $(Y,\rho)$ be metric spaces, $f:X\to Y$ a function between them and $\omega \in MC$. We will say that $f$ \textit{admits the modulus (of continuity)} $\omega$ if for any $x,y\in X$ the following condition is satisfied:
\begin{equation*}
    \rho(f(x),f(y)) \leq \omega(d(x,y)).
\end{equation*}
In particular, any function admitting a modulus of continuity is uniformly continuous and if $\omega'(0)\in\R_+$ then $f$ is $\omega'(0)$-Lipschitz (by (MC3) we have $\omega(t) \leq \omega'(0)t \ \forall t\in\R_+$). If $\F$ is a family of functions from $X$ to $Y$ then we will say analogously that the family $\F$ admits a modulus $\omega$ if every function from $\F$ admits that modulus. Note that if the family $\F$ admits a modulus of continuity then it is uniformly equicontinuous.\\
\indent For a set $X$, a family of functions $\F \subset X^X$ and a natural number $n > 0$, $\F^n$ will denote the set of all possible compositions of $n$ functions from the family $\F$; $\F^0:=\{id_X\}$. Moreover, $f^n$ will denote \textit{the nth iteration} of the function $f$, that is, $f^n:=f \circ \ldots \circ f$, where $f$ appears $n$ times.\\
\indent The proof of the following lemma is left to the reader.

\begin{lem}\label{rownosc a modul}
    Let $X$ be a topological space, $d$ and $\rho$ be equivalent metrics on a space $Y$ and $\F \subset Y^X$ be equicontinuous with respect to $\rho$. If $id_Y : (Y,\rho) \to (Y,d)$ is uniformly continuous then $\F$ is equicontinuous with respect to $d$.\\
    In particular, the family $\F$ is equicontinuous with respect to $\omega\circ\rho$ for any $\omega\in MC$ such that $\omega(x)\neq 0 \ \forall \ x \neq 0$.
\end{lem}


\section{Main result}

Theorem \ref{glowne_twierdzenie} has its origin in the following propositions.

\begin{obs}
    For a metrizable space $X$, metric space $(Y,\rho)$ and a family $\F$ of functions from $X$ to $Y$ the following conditions are equivalent:
    \begin{enumerate}
        \item $\F$ is equicontinuous with respect to $\rho$;
        \item there exists a compatible metric $d$ on $X$ with respect to which all functions from $\F$ are 1-Lipschitz.
    \end{enumerate}
\end{obs}

The situation changes drastically when $Y=X$ and we want to have the same metric in both domain and codomain.

\begin{obs}
    For a metrizable space $X$ and a family $\F$ of functions from $X$ to $X$ the following are equivalent:
    \begin{enumerate}
        \item there exists a compatible metric on $X$ such that the semigroup $\langle\F\rangle$ generated by $\F$ is equicontinuous.
        \item there exists a compatible metric on $X$ with respect to which all functions from $\F$ are 1-Lipschitz;
    \end{enumerate}
\end{obs}

Both of the above results are well known. Now, the question is -- can we do anything to control Lipschitz constants of functions from an equicontinuous family and their compositions? The answer is yes and Theorem \ref{glowne_twierdzenie} will tell us how we can do it.

Before stating and proving the main theorem, we need two lemmas and additional notions. The lemmas are likely well known. However, for the sake of completeness, we give their elementary proofs.

\begin{lem}\label{ciag}
    Let $\{a_n\}_{n=1}^{+\infty} \subset \R$ be a sequence satisfying $a_n > 1$ for all $n>0$ and $a_n \to +\infty$. Then there exists a sequence $\{b_n\}_{n=1}^{+\infty} \subset \R$ such that:
    \begin{itemize}
        \item $1 < b_n \leq a_n$ for all $n>0$,
        \item $b_n \to +\infty$,
        \item $b_{n+m} \leq b_n b_m$ for all $n,m>0$, that is, the sequence $\{b_n\}_{n=1}^{+\infty}$ is \textbf{submultiplicative}.
    \end{itemize}
\end{lem}

\begin{proof}
    The assumptions imply that $1 < c:=\inf\{a_n : n > 0\}$. We define inductively $i_0:=1$ and for $\nu\geq 1$ we put $i_\nu:=\max\{2\cdot i_{\nu-1},\min\{n > 0 : \ c^\nu \leq a_m \ \forall \ m \geq n \ \}\}$, an index starting from which all elements of the sequence are greater than or equal to $c^\nu$and which is at least twice as big as the previous one. For $n, \nu \in \N$ such that $i_\nu \leq n < i_{\nu+1}$ we put $b_n:=c^\nu$. Observe that such a sequence satisfies the first two conditions; what is left is to show it is submultiplicative. We will prove it by induction with respect to $n+m$.\\
    If $n+m=2$ then the condition is satisfied.\\
    Assume the condition holds for $n+m=N$ where $N \geq 2$. Consider $n$ and $m$ such that $n+m=N+1$. Then $1 < n$ or $1 < m$; without loss of generality $m \leq n$ (so $n > 1$).\\
    If $b_{N+1} = b_N$ then by induction hypothesis $b_{N+1} = b_N\leq b_m b_{n-1} \leq b_m b_n$.\\
    In the other case $b_N < b_{N+1}$ so $b_N = c^{\nu-1}$ and $b_{N+1}=c^\nu$ for some $\nu > 0$ and thus $N+1 = i_\nu$. Then $i_{\nu-1} \leq \frac{i_\nu}{2} = \frac{N+1}{2} = \frac{m+n}{2} \leq n \leq N < i_\nu$. Hence $i_{\nu-1} \leq n < i_\nu$ so $b_{N+1} = c\cdot c^{\nu-1} = cb_n \leq b_m b_n$.
\end{proof}

The next lemma is regarding the well-known \textit{tent map} (with parameter 2) $T:[0,1] \to [0,1]$ with $T(x):=1-|2x-1|$.

\begin{lem}\label{trojkat}
    For any $n>0$ the function $T^n$ is \textbf{not} 1-Lipschitz with respect to any compatible metric on $[0,1]$.
\end{lem}

\begin{proof}
    Assume, on the contrary, that there are a natural number $n > 0$ and a metric $d$ on $[0,1]$ compatible with the natural topology such that $T^n$ is 1-Lipschitz. Then for any $m>0$ the function $T^{nm}$ is also 1-Lipschitz so $$d(0,1)=d(T^{nm}(\frac{1}{2^{nm}}),T^{nm}(\frac{1}{2^{nm-1}})) \leq d(\frac{1}{2^{nm}},\frac{1}{2^{nm-1}}).$$ On the other hand $d(\frac{1}{2^{nm}},\frac{1}{2^{nm-1}}) \to 0 \ (m \to +\infty)$ which leads to a contradiction.
\end{proof}


\begin{df}\label{skladanie}
    Let $(X,d)$ be a metric space. A family $\F \subset X^X$ is called \textit{finitely equicontinuous} iff $\F^n$ is equicontinuous for any $n>0$.
\end{df}

In particular, a uniformly equicontinuous family is finitely equicontinuous. Indeed, let $(X,d)$ be a metric space and $\F \subset X^X$ a uniformly equicontinuous family of functions. We will prove that $\F^n$ is uniformly equicontinuous for any $n>0$. We will do it by induction. The base case, $n=1$, is obvious. Suppose $\F^n$ is uniformly equicontinuous for some $n>0$. Fix $\varepsilon>0$. Then there are $\delta_1, \delta_2>0$ such that for any $x,y \in X$ we have
     \begin{align*}
        d(x,y) < \delta_1 & \implies d(f_1(x),f_1(y))<\varepsilon \\
        d(x,y) < \delta_2 & \implies d(f_2(x),f_2(y))<\delta_1
     \end{align*}
     for any $f_1 \in \F^n$ and $f_2 \in \F$. Since any $f \in \F^{n+1}$ is of the form $f_1 \circ f_2$ for some $f_1 \in \F^n$ and $f_2 \in \F$ we obtain uniform equicontinuity for $\F^{n+1}$ with $\delta = \delta_2$.
\begin{df}
    We will call a modulus of continuity $\omega$ a \textit{simple modulus of continuity} if there exist $\alpha \in \R_+$ and $\beta \in \R_+\cup\{+\infty\}$ such that $\omega(t)=\min\{\alpha t,\beta\} \ \forall \ t \in \R_+$. In particular, when $\beta > 0$ we have $\omega'(0)=\alpha$.
\end{df}

 To our best knowledge the following theorem is not known.

\begin{tw}\label{glowne_twierdzenie}
    For a sequence $\{\omega_n\}_{n=1}^{+\infty} \subset MC$ the following conditions are equivalent:
    \begin{enumerate}
        \item For any metrizable space $X$ and any family $\F \subset  X^X$ which is finitely equicontinuous with respect to some compatible metric on $X$, there exists a compatible metric $d$ on $X$ such that for any $n>0$ the family $\F^n$ admits the modulus $\omega_n$.
        
        \item There exists a compatible metric $d$ on $[0,1]$ such that for any $n>0$ the function $T^n$ admits the modulus $\omega_n$, where $T$ is the tent map.
        
        \item There exists a sequence $\{\phi_n\}_{n=1}^{+\infty}$ of simple moduli of continuity such that:
            \begin{itemize}
                \item for any $n>0$ we have $ 1 < \phi_n'(0)$ and $\phi_n'(0) \to +\infty \ (n \to + \infty)$,
                \item there exists a constant $c > 0$ such that for any $n>0$ we have $\sup{\phi_n}>c$,
                \item for $n>0$ and any $t \in \R_+$ we have $\phi_n(t) \leq \omega_n(t)$.
            \end{itemize}
       
        \item There exists a constant $c > 0$ such that for any real number $t> 0$ we have $c < \liminf\limits_{n \to +\infty}{\omega_n}(t)$ and $1 < \omega_n'(0)$ for any $n > 0$.
       
    \end{enumerate}
\end{tw}

\begin{proof}
    (i) $\Rightarrow$ (ii) : Simply take $X=[0,1]$ and $\F=\{T\}$. Then $\F$ is finitely equicontinuous (since $\F$ is finite and $T$ is continuous) and $\F^n=\{T^n\}$.
    
    (ii) $\Rightarrow$ (iv) : Let $d$ be a metric from (ii). Put $t_n:=d(\frac{1}{2^{n-1}},\frac{1}{2^n})$ for $n > 0$. Since the function $T^n$ admits the modulus $\omega_n$, the inequality $d(0,1) = d(T^n(\frac{1}{2^{n-1}}),T^n(\frac{1}{2^n})) \leq \omega_n(t_n)$ holds. Moreover, $t_n \to 0 \ (n \to +\infty)$. Fix $t > 0$. There exists a natural number $N$ such that for any $n > N$ we have $t_n < t$ which gives $d(0,1) \leq \omega_n(t)$ by the condition (MC2); hence $0 < \frac{d(0,1)}{2} < \liminf\limits_{n \to +\infty}{\omega_n}(t)$.\\
    \indent By Lemma \ref{trojkat} none of the functions $T^n$ is 1-Lipschitz with respect to $d$ and since $T^n$ is $\omega_n'(0)$-Lipschitz whenever $\omega_n'(0)$ is finite, we get $1 < \omega_n'(0)$ for any $n >0$.
    
    (iv) $\Rightarrow$ (iii) : Let $c$ be the constant from (iv). For any $m \geq 1$ there exists a constant $N_m \in \N$, such that for $n \geq N_m$ we have $c < \omega_n(\frac{c}{m})$; we may assume that $N_m < N_{m+1}$. For $N_m \leq n < N_{m+1}$ we set $\phi_n(t):=\min\{\omega_n(\frac{c}{m})\frac{m}{c} \cdot t,c\}$; it follows that $\phi_n'(0) = \omega_n(\frac{c}{m})\frac{m}{c} > c \cdot\frac{m}{c} = m \geq 1$. Therefore $\phi'_n(0) > m$ for all $n \geq N_m$ which implies that $\phi_n'(0) \to +\infty \ (n \to + \infty)$. Moreover, since the moduli of continuity are nondecreasing and satisfy (MC3), we have $\phi_n \leq \omega_n$ and $\sup \phi_n = c$ for $n \geq N_1$. \\
    Now we consider the case when $n < N_1$. As the derivatives at 0 are greater than 1 and there are finitely many $n< N_1$ we may find an appropriate lower bound with cut moduli of continuity. Namely, for fixed $n < N_1$ we may find $t_n\in\R_+$ such that $\frac{\omega_n(t_n)}{t_n} > 1$ and put $\phi_n(t):=\min\{\frac{\omega_n(t_n)}{t_n}\cdot t, \omega_n(t_n)\}$ which satisfies the conditions. Received sequence fulfills the conditions from (iii) with constant not necessarily $c$ but with $\tilde{c}:=\frac{1}{2}\min\limits_{n>0}\{\sup \phi_n\} > 0$.
    
    
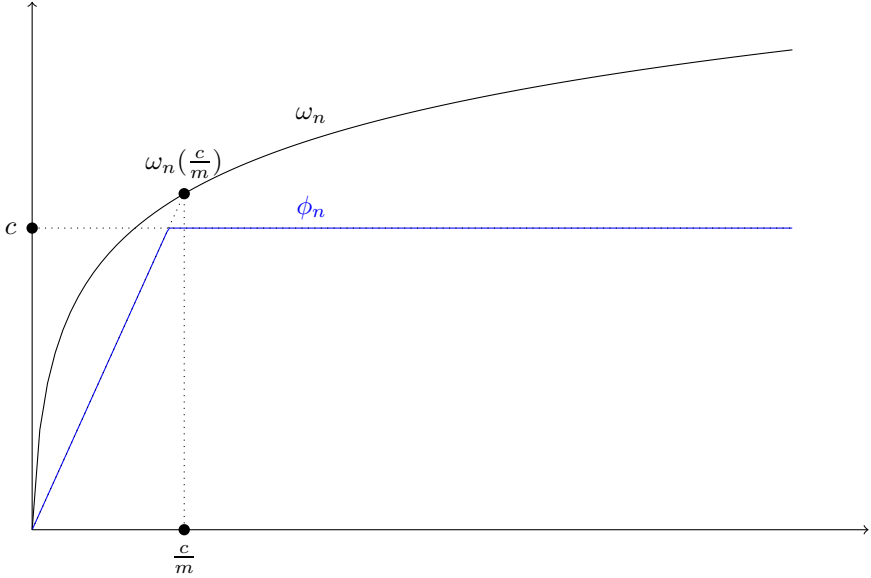
\begin{center}
        \begin{tikzpicture}[scale=1]
            \pgfmathsetmacro{\m}{2}
            \pgfmathsetmacro{\g}{10}
            \pgfmathsetmacro{\c}{4}
         
            \draw[->] (0,0) -- (1.1*\g,0) node[right] {};
            \draw[->] (0,0) -- (0,7) node[above] {};
        
            \pgfmathdeclarefunction{f}{1}{%
        		\pgfmathparse{1.2*ln(20*#1+1)}
        	}
            
            \draw[domain=0:\g,samples=100,variable=\x,black] plot ({\x},{f(\x)});
            \draw[domain=0:\g,samples=300,variable=\x,blue] plot ({\x},{min(f(\c/\m)/\c*\m*\x,\c)});
        
        	\draw[dotted] (\c/\m,0) -- (\c/\m,f(\c/\m);
            \draw[dotted] (0,\c) -- (\g,\c);
            \draw[dotted] (0,0) -- (\c/\m,f(\c/\m);
        
            \node[circle,fill,inner sep=1.5pt,label={below :$\frac{c}{m}$}] at (\c/\m,0) {};
            \node[circle,fill,inner sep=1.5pt,label={left :$c$}] at (0,\c) {};
            \node[circle,fill,inner sep=1.5pt,label={above :$\omega_n(\frac{c}{m})$}] at (\c/\m,f(\c/\m) {};
        
            \node[black,above left] at (0.4*\g,f(0.4*\g) {$\omega_n$};
            \node[blue,above left] at (0.4*\g,\c) {$\phi_n$};
            
        \end{tikzpicture}
    \end{center}
    
    (iii) $\Rightarrow$ (i) : From Lemma \ref{ciag} applied to the sequence $\{\phi_n'(0)\}_{n=1}^{+\infty}$ we may find another sequence $\{b_n\}_{n=1}^{+\infty}$ such that $1<b_n\leq \phi_n'(0)$, $b_n \rightarrow +\infty$ and $b_{n+m} \leq b_nb_m$ for $n,m > 0$; additionally we put $b_0:=1$.
    By Lemma \ref{rownosc a modul} we may find a compatible metric $d$ on $X$ such that $\diam_d(X) < c$ and the family $\F$ is finitely equicontinuous with respect to that metric (simply take $\omega(t) := \min\{\frac{c}{2},t\}$). We define a new metric:
    \begin{align*}
        \hat{d}(x,y):=\sup\Big\{ \frac{d(f(x),f(y))}{b_n} : n \geq 0, f \in \F^n\Big\}.
    \end{align*}
    One can easily see that it is indeed a metric. Observe that $\diam_{\hat{d}}(X) \leq c$. We have to show it is equivalent to $d$ and that it satisfies the condition given in (i).
    
    Note that $d \leq \hat{d}$ so it is sufficient to show that $\hat{d}(x_n,x) \to 0$ whenever $d(x_x,x) \to 0$. To this end, assume $d(x_n,x) \to 0$ and fix $\varepsilon > 0$. Since $b_n \to +\infty$, there exists $N \in \N$ such that for $m \geq N$ we have $\diam_d(X) \leq b_m\varepsilon$.
    The families $\F^n$ are equicontinuous for $0 < n < N$ with respect to $d$ and thus $\F \cup \F^2 \cup \ldots \cup \F^{N-1}$ is also equicontinuous. Hence we may find $\delta > 0$ such that if $d(x_n,x) < \delta$ then $d(f(x_n),f(x)) < \varepsilon$ for $f \in \F \cup \F^2 \cup \ldots \cup \F^{N-1}$. Since $b_m \geq 1$ for any $m$, $\hat{d}(x_n,x) \leq \varepsilon$ whenever $d(x_n,x) < \delta$. This ends the proof of the equivalence of the metrics.
    
    Let $f \in \F^m$ for some $m$. For any $x,y \in X$ we get
    \begin{align*}
        \hat{d}(f(x),f(y)) = & \ \sup\Big\{\frac{d((g\circ f)(x),(g\circ f)(y))}{b_n} : n \geq 0, g \in \F^n\Big\}\\
        = & \ b_m\sup\Big\{\frac{d((g\circ f)(x),(g\circ f)(y))}{b_m b_n} : n \geq 0, g \in \F^n\Big\}\\
        \leq & \ b_m\sup\Big\{\frac{d(g(x),g(y))}{b_{m+n}} : n \geq 0, g \in \F^{m+n}\Big\}\\
        = & \ b_m\sup\Big\{\frac{d(g(x),g(y))}{b_n} : n \geq m, g \in \F^{n}\Big\}\\
        \leq & \ b_m\hat{d}(x,y).
    \end{align*}
    
    Moreover $\hat{d}(f(x),f(y)) \leq c$ so
    $$\hat{d}(f(x),f(y)) \leq \min\{b_m\hat{d}(x,y), c\} \leq \min\{\phi_m'(0)\hat{d}(x,y),\sup \phi_m\} = \phi_m(\hat{d}(x,y))$$ 
    since $b_m \leq \phi_m'(0)$ and $c < \sup \phi_n$. Thus the function $f$ admits the modulus $\phi_m$ and so it admits the modulus $\omega_m$ which finishes the proof. \\
\end{proof}

\begin{rem}\label{zwarta}
    Note that if the space $X$ is compact and a family $\F \subset X^X$ is equicontinuous, then $\F$ is uniformly equicontinuous and hence finitely equicontinuous. Thus, under this extra condition of compactness of $X$, the condition (i) of the previous result could be accordingly modified.
\end{rem}

\begin{rem}
    Observe that the lower limit of a sequence of quasiconcave functions is quasiconcave. Thus the first condition (involving constant $c$) of item (iv) of Theorem \ref{glowne_twierdzenie} is equivalent to the statement that the lower limit of a given sequence of moduli of continuity is not a modulus of continuity.
\end{rem}

Setting $\omega_n(t) = t \log(n+2)$, we obtain from Theorem \ref{glowne_twierdzenie} the results announced in Introduction.

\begin{obs}
    It follows from the proof of Theorem \ref{glowne_twierdzenie} that for any family $\F$ of functions on a metrizable space $X$ the following conditions are equivalent:
    \begin{enumerate}
        \item there exists a compatible metric $d$ such that the family $\F$ is finitely equicontinuous;
        \item there exists a compatible metric $d$ such that the family $\F$ is uniformly equicontinuous;
        \item for $\F$ we have the rule specified in item (i) of Theorem \ref{glowne_twierdzenie} of admitting the moduli of continuity satisfying the condition (iv) therein with respect to some compatible metric.
    \end{enumerate}

     In particular, the above rules apply to finite families of continuous functions.
\end{obs}

\begin{proof}
    (ii) $\Rightarrow$ (i) : Obvious. \\
    (i) $\Rightarrow$ (iii) : Comes from the equivalence of conditions (i) and (iv) in Theorem \ref{glowne_twierdzenie}. \\
    (iii) $\Rightarrow$ (ii) : If family $\F$ admits a modulus of continuity then it is uniformly equicontinuous.
\end{proof}

One may wonder when a family of functions is finitely equicontinuous with respect to some compatible metric. As mentioned in Remark \ref{zwarta}, equicontinuity and finite equicontinuity are equivalent notions in the case of compact spaces. Hence, the problem is interesting only in the noncompact case. The example given below gives some insight into this issue.

In the following example, $\N_1$ denotes $\N\setminus\{0\}$.

\begin{prz}\label{przykladzik}
    For each $k \in \N$ we will construct a family $\F_k \subset X^X$ where $X=\N_1\times\R$ such that $\F_k^k$ is equicontinuous with respect to some metric but $\F_k^{k+1}$ is not equicontinuous with respect to any compatible metric. In particular, $X$ is locally compact, locally connected and separable.

    Fix $k\geq 2$. For $n,m\in \N_1, n<m$ we define functions:
   \begin{itemize}
        \item
            $f_n(a,b) = \left\{
                \begin{array}{ll}
                    (a,n\cdot b) & a=n, \\
                    (a,b) & a\neq n;
                \end{array}
            \right.$
        \item
            $f_{n,m}(a,b) = \left\{
                \begin{array}{lll}
                    (m,b\cdot{m^{-\frac{k-2}{2}}}) & a=n, \\
                    (n,b\cdot{m^{-\frac{k-2}{2}}}) & a=m, \\
                    (a,b) & a \neq n,m.
                \end{array}
            \right.$
    \end{itemize}

    We will call functions $f_n$ ``dilations (at position $n$)'' and $f_{n,m}$ ``contractions (with transposition)''. Note that $f_1=id_X$.
    
    Define $\F_k:=\{f_n : n \in \N_1\} \cup \{f_{n,m} : n,m \in \N_1, n<m\}$.

    Observe that $g_m:=f_{1,m} \circ f_m^{k-1} \circ f_{1,m}$ belongs to $\F_k^{k+1}$ for $m>1$. Moreover, $g_m(\{1\}\times(-\delta,\delta)) = \{1\}\times(-m\cdot\delta,m\cdot\delta)$, for any $\delta>0$. Let $A=\{1\}\times(-1,1)$ and $x=(1,0)$. Fix a compatible metric $\rho$ and put $\varepsilon=\diam_\rho(A)>0$. It is clear that for any neighbourhood $U$ of $x$ there exists $m>1$ such that $A \subset g_m(U)$, so $\diam_\rho(g_m(U)) > \varepsilon$ which shows that $\F_k^{k+1}$ is not equicontinuous.

    Now introduce the following metric on $X$:

    \begin{align*} 
        d((a,b),(c,d))=\left\{
            \begin{array}{ll}
                \min\{\frac{1}{a},|b-d|\} & a=c \\
                |a-c| + \min\{\frac{1}{a},|b|\} + \min\{\frac{1}{c},|d|\} & a\neq c
            \end{array}
        \right.
    \end{align*}

    It is equivalent to $d_e|_X$, where $d_e$ denotes the usual metric on the plane. A verification that $d$ is a metric is left to the reader. Notice that $d \leq d_e$ on each vertical line.

    We will show that $\F_k^k$ is equicontinuous with respect to the metric $d$.

    Let $(a,b)\in X, \varepsilon>0$ and $N=\max\{a,\lceil\frac{1}{\varepsilon}\rceil\}$; observe that for $n\geq N$ we have $\diam_d(\{n\}\times\R) \leq \varepsilon$. Lastly, set $U = \{a\}\times(b-\delta,b+\delta)$, where $\delta = \frac{\varepsilon}{2\cdot N^k}$. It is a neighbourhood of $(a,b)$.
    
    Now suppose $f \in \F_k^k$, so $f = h_k \circ \ldots \circ h_1$ with $h_i \in \F_k$, where $i=1,\ldots,k$. Denote by $\mathcal{M}$ the set of all indices $m \in \N$ for which $h_i = f_m$ or $h_i = f_{n,m}$ for some $i \in \{1,\ldots,k\}$ (and $n < m$ if applicable) such that $(h_i \circ \ldots \circ h_0)(U) \neq (h_{i-1} \circ \ldots \circ h_0)(U)$, where $h_0 = id_X$. If $\mathcal{M}=\varnothing$ set $M=0$, otherwise set $M=\max\mathcal{M}$. In other words, $M$ is the highest index appearing among those functions $h_i$ which ``affect'' the set $U$.
    
    \begin{enumerate}
        \item If $f(U) \subset \{n\}\times\R$ for some $n \geq N$ we have $\diam_d(f(U)) \leq \varepsilon$.
        \item If $f(U) \subset \{n\}\times\R$ for some $n<N$ and $M > N$, we know that there are $i,j\in\{1,\ldots,k\}, \ i\neq j$, such that $h_i = f_{n_i,M}, h_j=f_{n_j,M}$ for some $n_i,n_j \in \N_1$ (from the definition of $\mathcal{M}$) -- i.e. there are at least 2 contractions among $h_i$'s (as $a < M$ and $n < M$). Since all dilations are at most at position $M$ we get $\diam_{d_e}(f(U)) \leq 2\cdot \delta \cdot M^{k-2} \cdot M^{-(k-2)} = 2 \cdot \delta \leq \varepsilon$, so $\diam_d(f(U)) \leq \varepsilon$.
        \item If $f(U) \subset \{n\}\times\R$ for some $n<N$ and $M \leq N$ we know that all dilations take place at most at position $N$, so $\diam_{d_e}(f(U)) \leq 2 \cdot \delta \cdot N^k = \varepsilon$, so $\diam_d(f(U)) \leq \varepsilon$.
    \end{enumerate}
    
    This exhausts all the possibilities and thus proves that $\F_k^k$ is equicontinuous with respect to $d$.
    
    Obviously, we may take $\F_1:=\F_2^2$ and $\F_0:=\F_2^3$.
\end{prz}

    As the example shows, finite equicontinuity of a family $\F$ with respect to some metric cannot be deduced from equicontinuity of $\F^n$ in some (other) metric, even if $n$ is high. We leave the following problem open:

\begin{problem}
    Let $X$ be a metrizable space and $\F \subset X^X$. Suppose $\F^n$ is equicontinuous with respect to a (compatible) metric $d_n$ for any $n \in \N_1$. Does there exist a (compatible) metric $d$ such that $\F$ is finitely equicontinuous?
\end{problem}


\end{document}